\documentclass[reqno,12pt]{article}

\usepackage{amsmath}
\usepackage{amsfonts}
\usepackage{amssymb}
\usepackage{eufrak}
\usepackage{amscd}
\usepackage{amsthm}
\usepackage{epsfig}
\usepackage{amstext}
\usepackage[all]{xy}

\theoremstyle{plain}

\newtheorem{theorem}{Theorem}

\newtheorem{lemma} [theorem] {Lemma}
\newtheorem{remark}[theorem]{Remark}
\newtheorem{definition}[theorem]{Definition}

\def\diam{\hskip0.02cm{\rm diam}\hskip0.01cm}

\newcommand{\WOT}{{\rm WOT}}

\usepackage{graphicx}
\usepackage{amsmath}
\usepackage{fullpage}

\newtheorem{corollary}[theorem]{Corollary}

\newtheorem{proposition}[theorem]{Proposition}

\numberwithin{theorem}{section} \numberwithin{equation}{section}

\newcommand{\E}{\mathcal{E}}

\newcommand{\X}{\mathcal{X}}

\newcommand{\T}{\mathcal{T}}

\newcommand{\cH}{\mathcal{H}}

\newcommand{\M}{\mathcal{M}}

\newcommand{\s}{\subset}
\newcommand{\B}{\mathcal{B}}

\begin{document}

\title{Fixed points of holomorphic
transformations of operator balls}

\author{M.\,I.~Ostrovskii\\
Department of Mathematics and Computer Science\\
St. John's University\\
8000 Utopia Parkway\\
Queens, NY 11439\\
USA\\
e-mail: {\tt ostrovsm@stjohns.edu} \and\\
V.\,S.~Shulman\\
Department of Mathematics\\
Vologda State Technical University\\
15 Lenina street\\
Vologda 160000\\
RUSSIA\\
e-mail: {\tt shulman\_v@yahoo.com}
\and\\ L.~Turowska\\
Department of Mathematical Sciences\\
Chalmers University of Technology and University of Gothenburg\\
SE-41296, Gothenburg\\SWEDEN\\
e-mail: {\tt turowska@chalmers.se}}

\maketitle

\newpage

\tableofcontents

\noindent{\bf Abstract.} A new technique for proving fixed point theorems for families of holomorphic
transformations of operator balls is developed. One of these theorems is used to show that a bounded
representation in a real or complex Hilbert space is orthogonalizable or unitarizable (that is similar to an
orthogonal or unitary representation), respectively, provided the representation has an invariant indefinite
quadratic form with finitely many negative squares.
\bigskip

\bigskip


\section{Introduction}\label{S:Intro}

The main purpose of this paper is to prove the existence of fixed
points for some collections of holomorphic maps of the unit ball
$\B$ in the space $L(K,H)$ of operators from a finite-dimensional
Hilbert space $K$ to a Hilbert space $H$.  More precisely we
consider 1) groups of biholomorphic maps, and 2) finite
commutative families of holomorphic maps, and prove that in both
cases the existence of a fixed point is equivalent to the
existence of an invariant set separated from the boundary of $\B$.
For the first case, this result was obtained in \cite{OST}; the
advantage of the present approach is that we do not use the deep
theory of Shafrir \cite{Sha} replacing it by some simple general
observations from metric geometry.

The standard tool in the study of holomorphic maps of $\B$ is the
Carath\'eodory metric $\rho = c_{\B}$: it is known that
holomorphic maps do not increase $c_{\B}$. A similar metric $\rho$
can be defined in the case where Hilbert spaces $K$ and $H$ are
over reals. We treat both cases simultaneously dealing with
$\rho$-non-expansive maps. This allows us to obtain applications
to invariant subspaces and the orthogonalization problem in real
Hilbert spaces similar to those that were obtained in \cite{OST}
for the complex case. Namely it is proved that a bounded group of
$J$-unitary operators in a real Pontryagin space has a maximal
negative invariant subspace and that a bounded representation of a
group in a real Hilbert space that preserves a sesquilinear form
with a finite number of negative squares is similar to an
orthogonal representation. We also remove the assumption of
separability of $H$ made in \cite{OST}.

Our main technical tool is the proof of a kind of normal structure
in $\B$ - the existence of non-diametral points for a sufficiently
rich class of sets. In more detail: first we establish that each
pair $X,Y\in \B$ has at least one {\it midpoint}, that is, a point
$Z\in\B$ satisfying $\rho(Z,W)\le 1/2(\rho(X,W)+\rho(Y,W))$ for
all $W\in \B$. We call a subset of $\B$ {\it mid-convex} if with
any two points it contains all of their midpoints. \WOT-compact
mid-convex subsets of $\B$ are called {\it m-sets}. Then we prove
that each $m$-set either is a singleton or has a non-diametral
point (Theorem \ref{T:non-diam}).

The next step is to show that the presence of a $\rho$-bounded orbit implies the existence of an invariant
$m$-set. After that we use \WOT-compactness and get the existence of a minimal invariant $m$-set. Then we apply
Theorem \ref{T:non-diam} to prove that such minimal set should be a singleton.

The described general scheme for proving fixed point theorems for
nonexpansive mappings goes back to \cite{BM48}. This scheme was
gradually developed by many different authors, and a large variety
of realizations of it was found in further work on fixed points of
nonexpansive mappings, see \cite{GR84}, \cite{Kir65},
\cite{Kir81}, \cite{R75}, \cite{R76}, numerous examples of
applications of this scheme can be found in \cite{KS01}.

Note that another instance of the connection between boundedness
of orbits and the existence of fixed points can be found in
\cite{R75}.

\section{Metric geometry: the midpoint property, fixed
po\-ints of groups isometries, and other tools}\label{S:metric}

Let $(\X,d)$ be a metric space. By a {\it ball} in $(\X,d)$ we
mean (unless it is explicitly stated otherwise) a closed ball
$E_{a,r} = \{x\in X: d(a,x)\le r\}$. We say that $(\X,d)$ is {\it
ball-compact} if a family of balls has non-void intersection
provided  each its finite subfamily has non-void intersection (see
\cite{Tak2}). It is easy to show that each ball-compact metric
space is complete.

A subset $M\subset \X$ is called {\it ball-convex} if it is an
intersection of a family of balls.  It is clear from the
definition that each ball-convex set is bounded and closed. The
compactness extends from balls to all ball-convex sets:

\begin{lemma}[\cite{Tak2}]\label{fip-conv}
Let $(\X,d)$ be ball-compact. A family $\{M_{\lambda}: \lambda\in \Lambda\}$ of ball-convex subsets of $\X$ has
non-empty intersection if each its finite subfamily has non-empty intersection.
\end{lemma}

\begin{proof}
For each $\lambda\in \Lambda$, we have $M_{\lambda} = \cap_{i\in I_{\lambda}}E_i^{\lambda}$ where all
$E_i^{\lambda}$ are balls. Let us consider the family of balls $\mathcal{U} = \{E_i^{\lambda}: \lambda\in
\Lambda, i\in I_{\lambda}\}$. Each finite subfamily of $\mathcal{U}$ has a non-void intersection because it
contains the intersection of corresponding $M_{\lambda}$. By ball-compactness, $\mathcal{U}$ has non-zero
intersection which clearly coincides with $\cap_{\lambda}M_{\lambda}$.
\end{proof}

The {\it diameter} of a subset $M\s \X$ is defined by
\begin{equation}\label{diam} \diam M = \sup\{d(x,y): x,y\in M\}. \end{equation}
A point $a$ in a bounded subset $M$ is called {\it diametral} if
$$\sup\{d(a,x):x\in M\} = \diam M.$$
A metric space $\X$ is said to have {\it normal structure} if
every ball-convex subset of $\X$ with more than one element has a
non-diametral point. A mapping $g:\X\to \X$ is called {\it
nonexpansive} if $d(g(x),g(y))\le d(x,y)$ for all $x,y\in\X$.
\medskip

The concept of normal structure, introduced by Brodskii and Milman
\cite{BM48} for Banach spaces, has  played a prominent role in
fixed point theory for nonexpansive mappings. See, for example,
\cite{GR84}, \cite{Kir65}, \cite{KS01}, \cite{R76}, \cite{R80}. In
particular the scheme of the proof of the following result is well
known; nevertheless the result itself seems to be new.

\begin{theorem}\label{fp-gen}
Suppose that a metric space $(\X,d)$ is ball-compact and has normal structure. If a group of isometries of
$(\X,d)$ has a bounded orbit, then it has a fixed point.
\end{theorem}

\begin{proof}
Let $G$ be a group of isometries of $(\X,d)$ and let $G(x)$ be a bounded orbit, where $x$ is some point in $\X$.
Then the family $\Phi$ of all balls containing $G(x)$ is non-empty. Since $G(x)$ is invariant under $G$, the
family $\Phi$ is also invariant: $g(E)\in \Phi$, for each $E\in \Phi$. Hence the intersection $M_1$ of all
elements of $\Phi$ is a non-void $G$-invariant ball-convex set.

Thus the family $\M$ of all non-void $G$-invariant ball-convex subsets of $\X$ is non-empty. It follows from
Lemma \ref{fip-conv} that the intersection of a decreasing chain  of sets in $\M$ belongs to $\M$. By Zorn
Lemma, $\M$ has minimal elements. Our aim is to prove that a minimal element $M$ of $\M$ consists of one point.

Assume the contrary and let $\diam M=\alpha>0$. Since $(\X,d)$ has normal structure, $M$ contains a non-diametral
point $a$. It follows that $M\subset\{x\in \X:~d(a,x)\le\delta\}$ for some $\delta<\alpha$. Set
$$O=\bigcap_{b\in M}E_{b,\delta}.$$
The set $O$ is non-empty because $a\in O$. Furthermore $O$ is ball-convex by definition. To see that $O$ is a
proper subset of $M$  take $b,c\in M$ with $d(b,c)
> \delta$, then $c\notin E_{b,\delta}$, hence $c\notin O$.

Since $G$ is a group of isometric transformations and $M$ is invariant under each element of $G$, the action of
$G$ on $M$ is by isometric bijections. Therefore $O$ is $G$-invariant. We get a contradiction with the
minimality of $M$.
\end{proof}

The following definition is essential for our results on normal
structure and non-diametral points.

\begin{definition} {\rm A metric space $(\X,d)$ is said to have
{\it the midpoint property} if for any two points $a,b\in \X$ there is $c\in \X$ such that
$$d(c,x)\le (d(a,x) +
d(b,x))/2 ~~~\forall ~x\in (\X,d).$$ Each point $c$ satisfying this condition is called a {\it midpoint} for
$(a,b)$, the set of all midpoints for $(a,b)$ is denoted by $m(a,b)$.}
\end{definition}

\begin{remark} {\rm A straightforward adaptation of the well-known argument of Menger \cite{Men28} (see
also \cite[Theorem 14.1]{Blu53}) shows that complete metric spaces
with the midpoint property have the {\it convex structure} defined
by Takahashi \cite{Tak1}, but we do not need this fact.}
\end{remark}

Our proof of the fact that operator balls have the midpoint
property uses the following definition and lemma. We say that an
isometric involution $\sigma$ of $\X$ is {\it a reflection in a
point $x_0\in \X$} if $d(x,\sigma(x)) = 2 d(x,x_0)$ for each $x\in
\X$.

\begin{lemma}\label{symmetry}
If there exists a reflection $\sigma$ in a point $x_0\in \X$, then
$x_0$ is a midpoint for each pair $(x,\sigma(x))$.
\end{lemma}

\begin{proof}
Let $x\in \X$. For each $y\in \X$, $d(y,x_0) = \frac{1}{2}
d(y,\sigma(y)) \le \frac{1}{2}(d(y,x) + d(\sigma(y),x))) =
\frac{1}{2}(d(y,x)+d(y,\sigma(x)))$. This is what we need.
\end{proof}

Our next result is the main technical tool for proving results on
normal structure for spaces with the midpoint property.

\begin{lemma}\label{perif}
Suppose that a metric space $(\X,d)$ has the midpoint property.
Let $M$ be a ball-convex subset of $\X$ and $\alpha$ be the
diameter of $M$. If all points of $M$ are diametral, then $M$
contains a net $\{c_{\lambda}: \lambda\in \Lambda \}$ with the
property: $\lim_{\lambda} d(c_{\lambda},x) = \alpha$ for each
$x\in M$.
\end{lemma}

\begin{proof}

For $a_1,...,a_n \in \X$, we denote by $m(a_1,...,a_n)$ the set of all points $a\in \X$ satisfying $d(a,x)\le
(d(a_1,x) +...+d(a_n,x))/n$ for each $x\in \X$. The easy induction argument shows that if $\X$ has the midpoint
property and $n = 2^k$, then $m(a_1,...,a_n)$ is nonempty.

Let $F = (a_1,...,a_n)$ be a finite subset of $M$ and let
$\varepsilon > 0$. We will show that there is a point $c =
c(F,\varepsilon)$ such that $d(a_i,c)\geq (1 - \varepsilon)\alpha$
for all $i$. It is easy to see that a set with the midpoint
property is either a singleton or it is infinite. Adding new
points we may assume that $n = 2^k$. Let $b\in m(a_1,...,a_n)$. If
$F\subset E_{a,r}$ then $d(b,a) \le (d(a_1,a)+...+d(a_n,a))/n \le
r$ so $b\in E_{a,r}$. Thus $b$ belongs to each ball that contains
$M$; by ball-convexity, $b\in M$.

Since $b$ is diametral, there is $c\in M$ with
 $d(b,c) \ge (1-
\frac{\varepsilon}{n})\alpha$. It follows that $(1-
\frac{\varepsilon}{n})\alpha \le \frac{1}{n}\sum_{k=1}^n
d(a_k,c)$. If $d(a_j,c) < (1-\varepsilon)\alpha$, for some $j\le
n$, then $$\frac{1}{n}\sum_{k=1}^n d(a_j,c)<
\frac{1}{n}(1-\varepsilon)\alpha +\frac{n-1}n\alpha=(1-
\frac{\varepsilon}{n})\alpha,$$ a contradiction. Hence $d(a_j,c)
\ge (1-\varepsilon)\alpha$ for $j\le n$.

Let now $\Lambda$ be the set of all pairs $(F,\varepsilon)$ with
the order given by: $(F_1,\varepsilon_1) \prec (F_2,
\varepsilon_2)$ if $F_1\subset F_2, \varepsilon_1 >
\varepsilon_2$. For $\lambda = (F,\varepsilon)\in \Lambda$, we
write $c_{\lambda} = c(F,\varepsilon)$. It is clear that
$\lim_{\lambda} d(c_{\lambda},x) = \alpha$ for each $x\in M$.
\end{proof}

The class of ball-convex sets is not sufficiently rich for our
needs, because it is not closed with respect to unions of
increasing families of sets. For this reason we introduce the
following class of sets, which does not have this drawback.

\begin{definition}
{\rm A subset $M$ of a metric space is called {\it mid-convex} if
$m(x,y)\subset M$ for each pair $(x,y)$ of points in $M$.}
\end{definition}

Clearly each ball is mid-convex. Since the class of all mid-convex
sets is closed under intersection, each ball-convex set is
mid-convex. It is also clear that each mid-convex subset $M$ of a
metric space with the midpoint property is itself a metric space
with the midpoint property. This shows that  Lemma \ref{perif}
admits the following extension:

\begin{lemma}\label{perif1}
Let $M$ be a bounded mid-convex subset of a metric space $(\X,d)$
having the midpoint property. If all points of $M$ are diametral,
then $M$ contains a net $\{c_{\lambda}: \lambda\in \Lambda \}$
with the property: $\lim_{\lambda} d(c_{\lambda},x) = \diam {M}$
for each $x\in M$.
\end{lemma}

\section{Fixed points for groups of biholomorphic maps on the operator ball}

Let $K,H$ be real or complex Hilbert spaces; by $L(K,H)$ we denote
the Banach space of all bounded linear operators from $K$ to $H$.
We denote the open unit ball of $L(K,H)$ by $\B$ and call it {\it
operator ball}. The main goal of this section is to show that if
$K$ and $H$ are complex spaces and $\dim K<\infty$, then the
operator ball possesses a metric $\rho$ which is invariant with
respect to biholomorphic transformations of $\B$ and satisfies the
conditions of Theorem~\ref{fp-gen}, and therefore any group of
biholomorphic automorphisms of the operator ball with a bounded
orbit has a fixed point.

\subsection{M\"obius transformations and invariant distance in the operator ball}\label{S:distance}

For each $A \in \B$, we define a transformation $M_A$ of $\B$ ({\it a M\"obius transformation}) setting
\begin{equation}\label{mobius}
M_A(X) = (1-AA^*)^{-1/2}(A+X)(1+A^*X)^{-1}(1-A^*A)^{1/2}
\end{equation}
(this definition  is due to Potapov \cite{Pot55}). It can be proved that
\begin{equation}\label{conv}M_A^{-1} = M_{-A}\end{equation} (see \cite{harris}, Theorem
2).\medskip

We set
\begin{equation}\label{dist0}
\rho(A,B) = \tanh^{-1}( ||M_{-A}(B)||).
\end{equation}
It is easy to see that if $K,H$ are complex, then $\rho$ coincides
with the Carath\'eodory distance $c_{\B}$ in $\B$. Indeed, by
\cite[Theorem 4.1.8]{Vesent}, $c_{\B}(0,B) = \tanh^{-1}(\|B\|)$
(this holds for the unit ball of every Banach space). Furthermore
all $M_A$ are clearly holomorphic mappings of $\B$, and the
formula (\ref{conv}) shows that the same is true for their
inverses. In other words, $M_A$ are biholomorphic automorphisms of
$\B$. Since $c_{\B}$ is invariant with respect to biholomorphic
automorphisms and $M_{-A}$ sends $A$ to $0$, we get:
\begin{equation}\label{dist}
c_{\B}(A,B) = \tanh^{-1} ||M_{-A}(B)|| = \rho(A,B).
\end{equation}

 The equality (\ref{dist}) shows that $\rho$ is a metric on
 $\B$.\medskip

To check that the same is true in the case of real scalars, one
can use the complexification. Indeed, setting $\tilde{H} = H
\oplus iH$, $\tilde{K} = K\oplus iK$ one defines maps $A\to
\tilde{A}$ from $L(H)$ to $L(\tilde{H})$ (and from $L(K)$ to
$L(\tilde{K})$, from $L(K,H)$ to $L(\tilde{K},\tilde{H})$, from
$L(H,K)$ to $L(\tilde{H},\tilde{K})$) by $\tilde{A}(x\oplus iy) =
Ax\oplus iAy$. These maps are isometric *-homomorphisms of
algebras and modules whence $\widetilde{M_A(B)} =
M_{\tilde{A}}(\tilde{B})$ and $\rho(A,B) =
\rho(\tilde{A},\tilde{B})$.

\subsection{Midpoint property of $(\B, \rho)$}

Our next aim is to prove that the metric space $(\B,\rho)$ has the
midpoint property.

\begin{lemma}\label{two}
$\rho(A,-A) = 2\rho(0,A)$.
\end{lemma}
\begin{proof}
Let $A = UT$ be the polar decomposition of $A$ (so $T = |A| = (A^*A)^{1/2}$). Then
\[\begin{split} \rho(A,-A)& =
\tanh^{-1}\|M_A(A)\| =
\tanh^{-1}\|(1-AA^*)^{-1/2}2A(1+A^*A)^{-1}(1-A^*A)^{1/2}\|\\&=
\tanh^{-1}\|2A(1-A^*A)^{-1/2}(1+A^*A)^{-1}(1-A^*A)^{1/2}\|\\ & =
\tanh^{-1}\|U2T(1+T^2)^{-1}\|= \tanh^{-1}\|2T(1+T^2)^{-1}\| =
\|\tanh^{-1}(2T(1+T^2)^{-1})\|\\ & = \|2\tanh^{-1}(T)\| =
2\tanh^{-1}\|T\| = 2\tanh^{-1}\|A\| = 2 \rho(0,A)\end{split}\] (we
used the identity $\tanh(2x)=2\tanh x/(1+\tanh ^2x)$ and the fact
that $\|f(T)\| = f(\|T\|)$ if a function $f$ is non-decreasing on
the interval $[0,\|T\|]$ and $T$ is a positive operator).
\end{proof}

\begin{lemma}\label{simm}
For every $A,B\in \B$, there is an isometric transformation $\varphi$ of $(\B,\rho)$ with $\varphi(A) = -
\varphi(B)$.
\end{lemma}
\begin{proof}
It suffices to show this for $A = 0$, because the M\"obius transformations (which are isometries of $(\B,\rho)$)
act transitively on $\B$. Thus we have to find $C,D$ in $\B$ such that $M_C(-D) = 0$, $M_C(D) = B$. Take $D =
C$, then we have the equation $M_C(C) = B$. By the above calculations, this means that $U2T(1+T^2)^{-1} = B$,
where $UT = C$ is the polar decomposition. If $B = VS$ is the polar decomposition of $B$, then we take $U = V$
and $T = \tanh(\frac{1}{2}\tanh^{-1}(S))$.
\end{proof}

\begin{proposition}\label{middle}
The space $(\B,\rho)$ has the midpoint property.
\end{proposition}
\begin{proof}
Let $A,B\in \B$; we have to prove that there is $C\in \B$ such that $\rho(X,C)\le
\frac{1}{2}(\rho(X,A)+\rho(X,B))$ for each $X\in \B$. By Lemma \ref{simm}, we may assume that $B = -A$. The
involution $\sigma(X) = -X$ is clearly isometric; it follows from Lemma \ref{two} that it is a reflection in the
point $0$. Now by Lemma \ref{symmetry}, $0$ is a midpoint for the pair $(A,-A)$.
\end{proof}

\subsection{Topologies on $\B$ and ball-compactness}

A set in $\B$ is  bounded with respect to $\rho$ if it is
contained in some $\rho$-ball, this happens if and only if the set
is contained in a $\rho$-ball with center $0$. On the other hand,
it follows from the definition of $\rho$ that a $\rho$-ball with
center $0$ coincides with the closure of a multiple $r\B$ of the
operator ball $\B$, for some $r<1$. Therefore a set is bounded if
and only if it is separated from the boundary of $\B$.

The following lemma is a special case of a more general result proved in \cite[Theorem IV.2.2]{Vesent} (in the
case of real scalars one can use the complexification).

\begin{lemma}\label{twomet}
The metric $\rho$ is equivalent to the operator norm on any
$\rho$-bounded set.
\end{lemma}

Thus the topology induced by the metric $\rho$ on $\B$ is
equivalent to the norm topology.
\medskip

Another important topology on $\B$ is the weak operator topology
\WOT~(see \cite[p.~476]{DS58}). We collect the facts about the
\WOT~that we need in the following lemma.

\begin{lemma}\label{WOTballs}
{\rm (i)} Each $\rho$-ball $E_{A,r}$ is convex and \WOT-compact.

{\rm (ii)} If $\dim K < \infty$, then each norm-closed convex
subset of $\B$ is \WOT-compact.

{\rm (iii)} If $\dim K < \infty$, then all M\"obius
transformations are $\WOT$-continuous.
\end{lemma}

\begin{proof}
(i) By definition, each $\rho$-ball is the image of a $\rho$-ball with center $0$ under a M\"obius transform.
Since $0$-centered $\rho$-balls are the usual balls and since a M\"obius transform is a fractional-linear
transformation, the statement follows from the fact that images of balls under fractional-linear transformations
are convex and \WOT-compact (see, for example, \cite{KS95}). \smallskip

(ii) It is easy to check that if $K$ is finite-dimensional, then the norm topology of $L(K,H)$ coincides with
the strong operator topology while \WOT~coincides with the weak topology of $L(K,H)$. Since convex norm-closed
bounded subsets of a reflexive space are weakly compact, the result follows.
\medskip

(iii) The \WOT-continuity of a more general class of maps
(``fractional-linear maps with compact (1,2)-entries'') was
noticed and used by Krein in \cite{Krein}.
\end{proof}

\begin{corollary}\label{ball-comp} {\rm (i)} The metric space $(\B,\rho)$ is ball-compact.

{\rm (ii)} Ball-convex subsets of $(\B,\rho)$ are \WOT-compact.
\end{corollary}

\subsection{The normal structure of $(\B,\rho)$}\label{normal_section}

Recall that the normal structure means the existence of
non-diametral points in ball-convex sets. We establish the
existence of non-diametral points for a wider class of subsets of
$\B$.

\begin{definition}\label{m-set}
{\rm A subset $M$ of $\B$ is called an {\it $m$-set} if it is
mid-convex and \WOT-compact.}
\end{definition}

It follows from the results of the previous section that the class
of all $m$-sets contains the class of all ball-convex sets.
Furthermore, if $M$ is a $\rho$-bounded (= separated from the
boundary) subset of $\B$, then there is a smallest $m$-set $M_0$
containing $M$ (the intersection of all $m$-sets containing $M$).
We call $M_0$ {\it the $m$-span of} $M$.

\begin{theorem}\label{T:non-diam} Let $K$ be finite dimensional. Then any $m$-set in $\B$ either is a singleton or
has a non-diametral point.
\end{theorem}

\begin{proof}
Let $M$ be an $m$-set of $(\B,\rho)$ which is not a singleton. We
have to prove that $M$ contains a non-diametral point. \medskip

Assume the contrary, that is, all points in $M$ are diametral. Let $\alpha=\diam M>0$. Since $\B$ and therefore
$M$ has the midpoint property, we have, by Lemma \ref{perif1}, that there is a net $\{A_{\lambda}\}$ in $M$ such
that $\lim_{\lambda} \rho(A_{\lambda},X) = \alpha$ for each $X\in M$.\medskip

Since $M$ is \WOT-compact, the net $\{A_{\lambda}: \lambda\in
\Lambda\}$ contains a weakly convergent cofinal subnet. To
simplify the notation we assume that the net $\{A_{\lambda}:
\lambda\in \Lambda\}$ itself \WOT-converges to some operator $W$.
Since $M$ is \WOT-compact, we have $W\in M$.
\medskip

Since $W\in M$, we get
\begin{equation}\label{E:limrho}\lim_{\lambda}\rho(W,A_{\lambda})=\alpha.
\end{equation}

We get a contradiction by proving
\begin{equation}\label{E:gamma}\sup_{\lambda,\mu}\rho(A_{\lambda},A_{\mu})>\alpha.\end{equation}

Since all M\"obius transformations are \WOT-continuous isometries
of $(\B,\rho)$ (see  Lemma \ref{WOTballs}(iii)), we may assume
without loss of generality that $W=0$ (otherwise we apply
$M_{-W}$).
\medskip

Let $\beta=\tanh\alpha$. Then \eqref{E:limrho} leads to $\lim_{{\lambda}}||A_{\lambda}||=\beta$ and it suffices
to show that
$$\sup_{{\lambda},{\mu}}||M_{A_{\mu}}(-A_{\lambda})||>\beta.$$

Since $K$ is finite dimensional and $A_{\lambda}\in L(K,H)$, we can select a strongly convergent subnet in the
net $\{A_{\lambda}^*A_{\lambda}\}$. So we assume that  $A_{\lambda}^*A_{\lambda}\to P$, where $P\in L(K,K)$. It
is clear that $P\ge 0$ and $\|P\| = \beta^2$.
\medskip

Choose $\varepsilon> 0$ and fix $\mu$ with $\|A_{\mu}^*A_{\mu} - P\|<\varepsilon $. For brevity, denote
$A_{\mu}^*A_{\mu}$ by $Q$. We prove that  $\lim_{{\lambda}}\|M_{A_{\mu}}(-A_{\lambda})\|
> \beta$ if $\varepsilon>0$ is small enough.
By the definition,
\begin{equation}\label{mobius2}
M_{A_{\mu}}(-A_{\lambda}) =
(1-A_{\mu}A_{\mu}^*)^{-1/2}(A_{\mu}-A_{\lambda})(1-A_{\mu}^*A_{\lambda})^{-1}(1-A_{\mu}^*A_{\mu})^{1/2}.
\end{equation}
Since $A_{\mu}^*$ is of finite rank, $A_{\mu}^*A_{\lambda} \to 0$
in the norm topology. Hence
$\lim_{{\lambda}}\|M_{A_{\mu}}(-A_{\lambda})\| =
\lim_{{\lambda}}\|T_{\lambda}\|$, where
$$T_{\lambda} = (1-A_{\mu}A_{\mu}^*)^{-1/2}(A_{\mu}-A_{\lambda})(1-A_{\mu}^*A_{\mu})^{1/2} =
A_{\mu} - (1-A_{\mu}A_{\mu}^*)^{-1/2}A_{\lambda}(1-A_{\mu}^*A_{\mu})^{1/2}.$$

It follows from the identity
$$(1-t)^{-1/2} - 1 = \frac{t}{(1-t)(1+(1-t)^{-1/2})}$$
that the operator $(1-A_{\mu}A_{\mu}^*)^{-1/2}$ is a finite rank perturbation of the identity operator. Since
$A_{\lambda}\to 0$ in \WOT , we obtain that $\|T_{\lambda}-S_{\lambda}\|\to 0$, where $S_{\lambda} = A_{\mu} -
A_{\lambda}(1-A_{\mu}^*A_{\mu})^{1/2}$.

Denote  $A_{\lambda}(1-A_{\mu}^*A_{\mu})^{1/2}$ by $B_{\lambda}$. Since $B_{\lambda}\to 0$ in \WOT , the
sequence
$$(A_{\mu} - B_{\lambda})^*(A_{\mu}-B_{\lambda}) - A_{\mu}^*A_{\mu} - B_{\lambda}^*B_{\lambda} = - A_{\mu}^*B_{\lambda} -
B_{\lambda}^*A_{\mu}$$  tends to zero in norm topology. Furthermore,
$$B_{\lambda}^*B_{\lambda} = (1-Q)^{1/2}A_{\lambda}^*A_{\lambda}(1-Q)^{1/2} $$ tends in norm topology to
$(1-Q)^{1/2}P(1-Q)^{1/2}$. Therefore
$$(A_{\mu} -
B_{\lambda})^*(A_{\mu}-B_{\lambda}) \to Q + (1-Q)^{1/2}P(1-Q)^{1/2}.$$ Since $\|P-Q\| <\varepsilon $, we have
that
$$\|Q + (1-Q)^{1/2}P(1-Q)^{1/2} - (Q +(1-Q)Q)\| <\varepsilon.$$ The inequalities  $$\beta^2 - \varepsilon\le
\|Q\|\le \beta^2$$ imply $$\|Q + (1-Q)Q\| \ge 2\beta^2 - \beta^4 - 2\varepsilon,$$ whence
$$\lim_{\lambda}||S_{\lambda}^*S_{\lambda}||=\lim_{{\lambda}} \|(A_{\mu} -
B_{\lambda})^*(A_{\mu}-B_{\lambda})\|\ge 2\beta^2 - \beta^4 -
3\varepsilon > \beta^2,$$ the last inequality is satisfied if
$\varepsilon$ is sufficiently small.
\end{proof}

Since $(\B,\rho)$ has the midpoint property (Proposition \ref{middle}), all ball-convex sets in $\B$ are
mid-convex and we obtain
\begin{corollary}\label{normal}
The metric space $(\B,\rho)$ has normal structure.
\end{corollary}

By Corollaries \ref{ball-comp}, \ref{normal}, and
Theorem~\ref{fp-gen}, we immediately get

\begin{theorem}\label{main}
Let $\dim K < \infty$. If a group $G$ of $\rho$-isometric maps of
$\B$  has at least one orbit separated from the boundary, then it
has a fixed point.
\end{theorem}

Since biholomorphic mappings of $\B$ are isometric with respect to
the Carath\'eo\-dory metric, we get also

\begin{corollary}\label{C:main}
Let $H$ and $K$ be complex, $\dim K < \infty$. If a group $G$ of biholomorphic automorphisms of $\B$ has at
least one orbit separated from the boundary, then it has a fixed point.
\end{corollary}

\section{Fixed points for commuting holomorphic maps on the operator ball}

We continue to suppose that $\dim K<\infty$. The spaces can be real or complex.

\begin{lemma}\label{invSet} Let $G$ be a commutative semigroup of $\rho$-nonexpansive maps of $\B$. If it has a
$\rho$-bounded orbit, then it has an invariant $m$-set.
\end{lemma}

\begin{proof} Let $X\in\B$ be an operator whose orbit $G(X)$ is $\rho$-bounded.
We introduce the following partial order on $G$: $g\succ h$ if and only if $g = hf$ for some $f\in G$. The
semigroup $G$ is a directed set with respect to this order (we use the standard definitions from \cite[Chapter
I]{DS58}).\medskip

Let $r=\diam(G(X))$ and let $F = \{Y\in \B: \limsup_{g\in (G,\succ)} \rho(Y,g(X)) \le r\}$. Then
$$F = \bigcap_{\varepsilon>0}\,\bigcup_{g\in G}\,\bigcap_{h\succ g}E_{h(X),r+\varepsilon}.$$

 Since each $\rho$-ball $E_{Z,p}$ is convex in the usual sense (see Lemma \ref{WOTballs}(i)) and is mid-convex, the same is
true for $F$. Also it is easy to see that $F$ is $\rho$-bounded and is closed in the topology defined by the
metric $\rho$. By Lemma \ref{twomet}, $F$ is norm-closed. By Lemma \ref{WOTballs}(ii), $F$ is WOT-compact, thus
$F$ is an $m$-set.

Let us check that $F$ is $G$-invariant. Indeed, if $Y\in F$ and
$h\in G$, then
\[\begin{split}\limsup_{g\in (G,\succ)} \rho(h(Y),g(X))& =
\limsup_{f\in (G,\succ)} \rho(h(Y),hf(X))\le \limsup_{f\in (G,\succ)} \rho(Y,f(X)) \le r.\end{split}\] This
means that $h(Y)\in F$.
\end{proof}

The idea of our proof of Lemma \ref{invSet} goes back to
\cite{Kir65}.

\begin{lemma}\label{oneMap}
 If a $\rho$-nonexpansive map $g:\B\to \B$  has an invariant bounded set $U$, then it has a fixed point in
$\B$.

Moreover, if  $E$ is an arbitrary invariant $m$-set containing
$U$, then a fixed point can be found in $E$.
\end{lemma}

\begin{proof} The assumption is equivalent to the condition that the semigroup generated by $g$
has a bounded orbit. By Lemma \ref{invSet}, there is an invariant
$m$-set for $g$. Since $m$-sets are \WOT-compact, there exists a
minimal invariant $m$-set $F$ for $g$. The $m$-span $F_1$ of
$g(F)$ is contained in $F$ and is invariant (indeed $g(F_1)\subset
g(F)\subset F_1$). Hence $F_1 = F$. If $F$ is a singleton, then it
is a fixed point.

Now assume that $F$ is not a singleton. By Theorem \ref{T:non-diam}, there is $r< \diam(F)$ such that the set
$$F(r):=\{Y\in F:~\forall Z\in F~ \rho(Y,Z)\le r \}$$ is non-empty. Since $F(r)$ is the intersection of $F$ with
a family of balls, it is an $m$-set. Moreover $F(r)$ is invariant. Indeed, if $Y\in F(r)$, then $\rho(g(Y),g(X))
\le r$ for all $X\in F$. Hence the ball $E_{g(Y),r}$ contains $g(F)$. Since $F$ is the $m$-span of $g(F)$, this
ball contains $F$. Hence $g(Y)\in F(r)$. It is also clear that $F(r)$ is a proper subset of $F$. We get a
contradiction with the minimality of $F$.

The second statement of the lemma follows from the construction.
\end{proof}

Using Lemma \ref{oneMap} we prove the following analogue of
Shafrir's theorem \cite[Theorem 3.2]{Sha92} for $\B$:

\begin{theorem}\label{common}
A finite set $\E$ of commuting $\rho$-nonexpansive maps of $\B$
has a fixed point in $\B$ if and only if $\E$ has an invariant set
separated from the boundary of $\B$.
\end{theorem}

\begin{proof} The ``only if'' part is trivial, we prove the ``if''
part only. Recall that a set is separated from the boundary if and only if it is $\rho$-bounded. If a set $\E =
\{g_1,...,g_n\}$ has a $\rho$-bounded invariant set, then the semigroup generated by $\E$ has a bounded orbit.
By Lemma \ref{invSet}, there is an invariant $m$-set for $\E$. Let $F$ be a minimal invariant $m$-set for $\E$.
\medskip

Let $g = g_1g_2...g_n$. By Lemma \ref{oneMap}, $g$ has a fixed point in $F$.
  Let $R$ be the set of all fixed
points of $g$ in $F$.

We claim that $g_i(R)= R$ for each $i$. The inclusion $g_i(R) \subset R$ follows from commutativity. On the
other hand, if $X\in R$, then $X = g(X) = g_i(Y)$, where $Y = h(X)$ and $h$ is the product of all $g_j$ with
$j\neq i$. Clearly $Y\in R$, $X\in g_i(R)$, thus, $R\subset g_i(R)$.
\medskip

Let $K$ be the smallest $m$-set containing $R$, clearly $K\subset
F$. If $K$ is a singleton, then $R=K$ contains only one point
which is a common fixed point for $g_1,\dots,g_n$.

If $K$ is not a singleton, then, by Theorem \ref{T:non-diam},
there is $r<\diam(K)$ such that the set $K_r = \{X\in K:
\rho(X,Y)\le r, \text{ for all }Y\in K\}$ is non-empty. Since
$K_r$ is an intersection of an $m$-set with a family of
$\rho$-balls, it is also an $m$-set.
\medskip

Let us show that $K_r$ is invariant under all $g_i$, $i=1,\dots,n$. First we note that $K_r = \{X\in K:
\rho(X,Y)\le r, \text{ for all }Y\in R\}$. Indeed, if $\rho(X,Y)\le r$ for all $Y\in R$, then $R\subset E_{X,r}$
whence $K\subset E_{X,r}$ because $E_{X,r}$ is an $m$-set. Now if $X\in K_r$, $Y\in R$, using the fact that
$g_i(R) = R$, we obtain $\rho(g_i(X),Y) = \rho(g_i(X),g_i(Y^{\prime})) \le \rho(X,Y^{\prime}) \le r$. By
minimality of $F$, we get $K_r = F$. We get a contradiction since $K\subset F$ and $K_r$ is a proper subset of
$K$.
\end{proof}

Note that Theorem \ref{common} extends to infinite families of
$\rho$-nonexpansive maps if these maps are assumed to be
\WOT-continuous (indeed in this case the set of all fixed points
for each finite family of maps is \WOT-compact).
\medskip

Since holomorphic maps are $\rho$-nonexpansive, Theorem
\ref{common} has the following corollary.

\begin{corollary}
Let $K$ and $H$ be complex spaces. A finite set $\E$ of commuting
holomorphic maps of $\B$ has a common fixed point in $\B$ if and
only if $\E$ has an invariant set separated from the boundary of
$\B$.
\end{corollary}

\section{Operator-theoretic applications: the case of real
spa\-ces}\label{S:oper}

For applications to invariant subspaces and orthogonalization, we
need to prove that $\rho$ is invariant under maps on $\B$ which
are induced by $J$-unitary operators of the space $\cH = H\oplus
K$. In the case of complex scalars this follows immediately from
the fact that these maps are holomorphic, while in the real case
we need some additional arguments.

We remind some definitions and constructions from the theory of
spaces with an indefinite product. We refer to \cite{AI89} or
\cite{KS97} for more information.

An operator $U$ on $\cH$ is called {\it $J$-unitary} if it preserves the indefinite scalar product
\begin{equation}\label{indscalprod}
[x,y] = (P_Hx,y) - (P_Kx,y),
\end{equation}
 where $P_H$ and $P_K$ are the projections on the summands in the
decomposition $\cH = H\oplus K$. To each $J$-unitary operator $U$ on $\cH$ there corresponds a map $w_U$ of $\B$
in such a way that
\begin{equation}\label{comp}
w_{U_1U_2} = w_{U_1}\circ w_{U_2}.
\end{equation}
To see this recall that a vector $x\in \cH$ is called {\it
positive} ({\it neutral, negative}) if $[x,x] > 0$ ($[x,x] = 0$,
$[x,x]  < 0$, respectively). A subspace of $\cH$ is called {\it
positive} ({\it neutral, negative}) if all its non-zero elements
are positive (neutral, negative, respectively). For each operator
$X\in \B$ the set $$S(X) = \{Xx\oplus x: x\in K\}$$ is a negative
subspace of $\cH$. Since $\dim(S(X)) = \dim(K)$, $S(X)$ is a
maximal negative subspace in $\cH$. Indeed, if some subspace $M$
of $\cH$ strictly contains $S(X)$, then its dimension is greater
than the codimension of $H$, whence $M\cap H \neq \{0\}$. But all
non-zero vectors in $H$ are positive.

Conversely, each maximal negative subspace $Q$ of $\cH$ coincides
with $S(X)$, for some $X\in \B$. Indeed, since $Q\cap H = \{0\}$,
there is an operator $X: K\to H$ such that each vector of $Q$ is
of the form $Xx\oplus x$. Since $Q$ is negative, we have
$[Xx\oplus x, Xx\oplus x]=\|Xx\|^2-\|x\|^2<0$. Since $K$ is finite
dimensional, this implies $\|X\| < 1$, so $X\in \B$. Thus
$Q\subset S(X)$; and, by maximality, $Q = S(X)$.

It is easy to see that the map $X\to S(X)$ from $\B$ to the set $\E$ of all maximal negative subspaces is
injective and therefore bijective.

Now we can define $w_U$. Note that if a subspace $L$ of $\cH$ is
maximal negative, then its image $UL$ under $U$ is also maximal
negative (because $U$ is invertible and preserves
$[\cdot,\cdot]$). Hence, for each $X\in \B$, there is $Y\in \B$
such that $S(Y) = US(X)$. We let $w_U(X)=Y$.

The equality (\ref{comp}) follows easily because
$$S(w_{U_1}(w_{U_2}(X))) = U_1S(w_{U_2}(X)) = U_1U_2S(X) =
S(w_{U_1U_2}(X))$$  and the map $X\to S(X)$ is injective.

Let $U = (U_{ij})_{i,j=1}^2 $ be the matrix of $U$ with respect to the decomposition $\cH = H\oplus K$. Then
$U(Xx\oplus x) = (U_{11}Xx + U_{12}x)\oplus(U_{21}Xx + U_{22}x)$. Since $U(Xx\oplus x) \in S(w_U(X))$, we
conclude that
$$w_U(X)(U_{21}Xx + U_{22}x) = U_{11}Xx + U_{12}x.$$
Thus
\begin{equation}\label{fraclin}
w_U(X) = (U_{11}X + U_{12})(U_{21}X + U_{22})^{-1}.
\end{equation}
Let us denote by $\T$ the group of all such transformations of
$\B$ (Helton \cite{Helt} calls them {\it general symplectic
transformations}). It should be noted that all M\"obius maps
belong to $\T$. Namely $M_A = w_U$ where $U$ is the $J$-unitary
operator with the matrix $(U_{ij})_{i,j=1}^2$, where
\begin{eqnarray*}
U_{11} =
(1_H-AA^*)^{-1/2}, U_{12} = - A(1_K-A^*A)^{-1/2},\\ U_{21} =
-A^*(1_H-AA^*)^{-1/2}, U_{22} = (1_K-A^*A)^{-1/2}.
\end{eqnarray*}
If a map $\varphi\in \T$ has the property $\varphi(0) = 0$, then it can be written in the form $\varphi(X) =
V_1XV_2$ where $V_1,V_2$ are unitary operators in $H$ and $K$, respectively. Indeed, let $U$ be a $J$-unitary
operator with $\varphi = w_U$, then $UK = US(0) = S(\varphi(0)) = S(0) = K$. Since $H$ is a
$[\cdot,\cdot]$-orthogonal complement of $K$ in $\cH$, it is also invariant under $U$. Thus $U$ has a diagonal
matrix with respect to the decomposition $\cH = H\oplus K$: $U_{12} = 0$, $U_{21}= 0$. A moment reflection shows
that $U_{11}$ and $U_{22}$ are unitary operators in $H$ and $K$, respectively. Thus $\varphi(X) =
U_{11}XU_{22}^{-1}$ is of the needed form.

It follows now that $\|\varphi(X)\| = \|X\|$ for all $X\in \B$, if $\varphi(0) = 0$.

If $\varphi\in \T$ sends $A$ to $0$, then the transformation
$\varphi_1 = \varphi\circ M_A$ preserves $0$ and sends $M_{-A}(B)$
to $\varphi(B)$. Thus, by the previous statement,  $\|\varphi(B)\|
= \|M_{-A}(B)\|$ and $\rho(A,B) = \tanh^{-1}\|\varphi(B)\| =
\rho(0,\varphi(B))$ for each $\varphi\in \T$ with $\varphi(A) =
0$.

This implies that $\rho(\varphi(A),\varphi(B)) = \rho(A,B)$ for
all $A,B\in \B$ and $\varphi\in \T$.  Indeed, let $\alpha\in \T$
with $\alpha(A) = 0$, then $\rho(A,B) = \rho(0,\alpha(B))$.
Setting $\psi = \alpha\circ \varphi^{-1}$ we get that
$\psi(\varphi(A)) = 0$ so $$\rho(\varphi(A),\varphi(B)) =
\rho(0,\psi(\varphi(B))) = \rho(0,\alpha(B)) = \rho(A,B).$$

Thus $\T$ is a group of isometries of $(\B,\rho)$. Applying
Theorem \ref{main} we obtain

\begin{corollary}\label{fixed-real}
If a group $G$ of general symplectic transformations  of $\B$ has an orbit separated from the boundary of $\B$,
then it has a fixed point in $\B$.
\end{corollary}

Recall that subspaces $M,N$ of a space with indefinite scalar product $\cH$ form {\it a dual pair of subspaces}
if $M$ is positive, $N$ is negative and $M+N = \cH$.

A {\it Pontryagin space} $\Pi_k$ is the space $\cH = H\oplus K$
with the scalar product (\ref{indscalprod}), where $H$ and $K$ are
Hilbert spaces, $\dim H = \infty$, $\dim K = k$.
\begin{corollary}\label{dualpair-real}
Any bounded group of $J$-unitary operators in a real Pontryagin
space $\Pi_k$ has an invariant dual pair of subspaces.
\end{corollary}
\begin{proof}
Let $\Gamma$ be a bounded group of $J$-unitary operators in $\cH$,
$C = \sup_{U\in \Gamma}\|U\|$ and $G = \{w_U: U\in \Gamma\}$. Then
the orbit $G(0)$ is separated from the boundary in $\B$. Indeed,
let $X = w_U(0)$, for some $U\in \Gamma$. For each $x\in K$, the
vector $Xx\oplus x$ can be written as $U(0\oplus y)$ where $y\in
K$. Hence $\|Xx\oplus x\|^2 \le C^2\|y\|^2$ and $[Xx\oplus x,
Xx\oplus x] = [0\oplus y,0\oplus y]$. Thus $\|Xx\|^2 -\|x\|^2 =
-\|y\|^2$ and $\|Xx\|^2 +\|x\|^2 \le C^2\|y\|^2$. It follows that
$\|Xx\|^2 \le r^2 \|x\|^2$ where $r = \sqrt{\frac{C^2-1}{C^2+1}} <
1$.  Thus $\|X\| \le r$ for all $X\in G(0)$.

By Corollary \ref{fixed-real}, $G$ has a fixed point $X\in \B$. Hence the maximal
negative space $N = S(X)$ is
invariant under $\Gamma$. Denoting by $M$ the $[\cdot,\cdot]$-orthogonal complement of $N$ it is easy to see
that $(M,N)$ is an invariant dual pair of subspaces.
\end{proof}

Now we turn to the problem of orthogonalization of a representation of a group on a Hilbert space (the standard
reference of the topic is the book \cite{Pis01}). As for representations in complex Hilbert spaces (see
\cite{OST}), we say that a representation $\pi$ of a group $\Gamma$ in a real Hilbert space $\cH$ preserves a
sesquilinear form with finite number of negative squares if $\cH$ can be decomposed into the orthogonal sums of
subspaces $H$ and $K$ (with $\dim K < \infty$) in such a way that $(P_H\pi(g)x,\pi(g)y) - (P_K\pi(g)x,\pi(g)y) =
(P_Hx,y) -(P_Kx,y)$ for all $x,y\in \cH$ and all $g\in G$.

\begin{corollary}\label{real-orthog}
If a bounded representation of a group $\Gamma$ by operators in a real Hilbert space $\cH$ preserves a
sesquilinear form with finite number of negative squares, then it is similar to an orthogonal representation of
$\Gamma$ on $\cH$.
\end{corollary}
\begin{proof}
By our assumptions, $\pi$ acts by $J$-unitary operators on a $\Pi_k$ space $H\oplus K$. Since $\Gamma_0 =
\pi(\Gamma)$ is a bounded group of $J$-unitary operators, the previous Corollary shows that there is an
invariant for $\Gamma_0$ dual pair $(M,N)$ of subspaces. Moreover $N = UK$ for some $J$-unitary operator $U$.
Then all operators $\tau(g) = U^{-1}\pi(g)U$ are $J$-unitary, and the subspace $K$ is invariant for them. It
follows that $H$ is also invariant for operators $\tau(g)$. Hence these operators preserve the scalar product on
$\cH$. Thus $g\mapsto \tau(g)$ is a unitary representation similar to $\pi$.
\end{proof}

{\bf Acknowledgement.} The authors are grateful to Professor S.
Reich for helpful suggestions and references.

\end{document}